\documentclass[10pt]{amsart}
\usepackage{amsfonts}
\usepackage{ifthen}
\usepackage{amsthm}
\usepackage{amsmath,mathrsfs}
\usepackage{graphicx}
\usepackage{amscd,amssymb,amsthm}
\usepackage{color}
\usepackage{hyperref}

\newcounter{minutes}
\divide\time by 60
\newcounter{hours}
\multiply\time by 60 \addtocounter{minutes}{-\time}

\newtheorem{theorem}{Theorem}

\newtheorem{corollary}{Corollary}
\newtheorem{remark}{Remark}

\title[Generalized Volterra functions]{Generalized Volterra functions, its integral representations and applications to the Mathieu--type series}

\author[K. Mehrez, S. M. Sitnik]{Khaled Mehrez\;and Sergei M. Sitnik}
\address{Khaled Mehrez\newline
D\'epartement de Math\'ematiques, Universit\'e de Kairouan, Tunisia \textit{and}\newline
D\'epartement de Math\'ematiques, Facult\'ee des sciences de Tunis, Universit\'e Tunis El Manar, Tunisia.}
\email{k.mehrez@yahoo.fr}

\address{Sergei M. Sitnik\newline
Belgorod State National Research University (BSU), Belgorod, Russia.}
\email{Sitnik@bsu.edu.ru}



\begin{document}

\def\thefootnote{}
\footnotetext{ \texttt{File:~\jobname .tex,
          printed: \number\year-0\number\month-\number\day,
          \thehours.\ifnum\theminutes<10{0}\fi\theminutes}
} \makeatletter\def\thefootnote{\@arabic\c@footnote}\makeatother

\maketitle

\begin{abstract}
In this paper we introduce the new class of generalized Volterra  functions. We prove some integral representations for them via Fox--Wright H--functions and Meijer G--functions. From  positivity conditions on the weight in these representations, we found sufficient conditions  on  parameters of the generalized Volterra  function to prove its complete monotonicity. As applications we prove a Tur\'an type inequality for generalized Volterra  functions and derive closed--form integral representations  for a family of convergent Mathieu--type series defined in terms of  generalized Volterra  functions.
\end{abstract}

\section{Introduction}
Consider definitions of classical Volterra and  related  functions  as stated in (\cite[p. 217]{ER}):
\begin{eqnarray}
\nu(x)&=&\int_0^\infty\frac{x^t}{\Gamma(t+1)}dt,\\
\nu(x,\alpha)&=&\int_0^\infty\frac{x^{t+\alpha}}{\Gamma(t+\alpha+1)}dt,\\
\mu(x,\beta)&=&\int_0^\infty\frac{x^t t^\beta}{\Gamma(t+1)\Gamma(\beta+1)}dt\\
\mu_(x,\beta,\alpha)&=&\int_0^\infty\frac{x^{t+\alpha} t^\beta}{\Gamma(t+\alpha+1)\Gamma(\beta+1)}dt,
\end{eqnarray}
where $\alpha,\beta>-1$ and $x>0,$ but some particular notations are usually adopted in  special cases
\begin{equation}
\begin{split}
\alpha&=\beta=0,\;\;\nu(x)=\mu(x,0,0)\\
\alpha&\neq0,\beta=0,\;\;\nu(x,\alpha)=\mu(x,0,\alpha)\\
\alpha&=0,\beta\neq0,\;\;\mu(x,\beta)=\mu(x,\beta,0).
\end{split}
\end{equation}

Volterra functions were introduced by Vito Volterra in 1916. Its theory was thoroughly developed by Mhitar M. Dzhrbashyan, his own and his coathors results were summed up in the monograph \cite{Dzh} in 1966. In this book many important results on Volterra functions, known and new, were gathered and introduced.  Many results on Volterra functions were also gathered in two books of A. Apelblat \cite{App1}--\cite{App2}, for important application cf. also \cite{GaMa}.

In this paper we define the new class of  generalized Volterra  functions $ V_{p,q}^{\alpha,\beta}[.]$  with $p$ numerator parameters $\alpha_1,...,\alpha_p$ and $q$ denominator parameters $\beta_1,...,\beta_q,$ by
\begin{equation}\label{DefGenVolt}
V_{p,q}^{\alpha,\beta}\Big[_{(b_1,B_1),...,(b_q,B_q)}^{(a_1,A_1),...,(a_p,A_p)}\Big|x \Big]=V_{p,q}^{\alpha,\beta}\Big[_{(\beta_q,B_q)}^{(a_p,A_p)}\Big|x \Big]=\int_0^\infty \frac{\prod_{i=1}^p\Gamma(A_it+a_i)}{\prod_{j=1}^q\Gamma(B_jt+b_j)}\frac{x^{t+\alpha}\ t^\beta}{\Gamma(\beta+1)} dt,
\end{equation}
where $$\Big(\alpha,\beta>-1,\;x>0,A_i,B_j>0,\;a_i,b_j\geq0,(i=1,...,p;j=1,...,q)\Big).$$

For the special case $x=1/e, \beta=s-1$ the function \eqref{DefGenVolt} up to the constant reduces to the function
\begin{equation}
V1_{p,q}^{\alpha,\beta}\Big[_{(\beta_q,B_q)}^{(a_p,A_p)}\Big|p,s \Big]=\int_0^\infty \frac{\prod_{i=1}^p\Gamma(A_it+a_i)}{\prod_{j=1}^q\Gamma(B_jt+b_j)}{e^{-pt}\ t^{s-1}} dt,
\end{equation}
which is interesting and important as a simultaneous expression for Laplace and Mellin transforms for the gamma--function ratio.

The paper is organized as follows. In section 2 we prove several integral representations for the new class of the  generalized Volterra function via  Fox--Wright functions and the Laplace transform. Various new facts regarding the generalized Volterra function are proved, including complete monotonicity property, log--convexity in upper parameters, and a Tur\'an type inequality. In section 3 closed--form integral expressions are derived  for a family of convergent Mathieu--type series and its alternating variant when terms contain the generalized Volterra function.

\section{Integral representations for the generalized Volterra functions}

To formulate our first main result we  need a particular case of Fox's H--function defined by
\begin{equation}
H_{q,p}^{p,0}\left(z\Big|^{(B_1,b_1),...,(B_q,b_q)}_{(A_1,a_1),...,(A_p,a_p)}\right)=H_{q,p}^{p,0}\left(z\Big|^{(B_q,b_q)}_{(A_p,a_p)}\right)=\frac{1}{2i\pi}\int_{\mathcal{L}}\frac{\prod_{i=1}^p\Gamma (A_i s+a_i)}{\prod_{j=1}^q\Gamma (B_k s+b_j)}z^{-s}ds,
\end{equation}
where $A_i, B_j>0$ and $a_i,b_j$ are real.  The contour $\mathcal{L}$ can be either the left loop $\mathcal{L}_-$ starting at $-\infty+i\alpha$ and ending at $-\infty+i\beta$  for some $\alpha<0<\beta$ such that all poles of the integrand lie inside the loop, or the right loop $\mathcal{L}_+$  starting $\infty+i\alpha$ at and ending $\infty+i\beta$ and leaving all poles on the left, or the vertical line $\mathcal{L}_{ic},\;\Re(z)=c,$ traversed upward and leaving all poles of the integrand on the left. Denote the rightmost pole of the integrand by $\gamma:$
$$\gamma=-\min_{1\leq i\leq p}(a_i/A_i).$$
Let
\begin{equation}\label{rho}
\rho=\left(\prod_{i=1}^p A_i^{A_i}\right)\left(\prod_{j=1}^q B_j^{-B_j}\right),\;\mu=\sum_{j=1}^qb_j-\sum_{i=1}^p\alpha_i+\frac{p-q}{2}.
\end{equation}
Existence conditions of Fox's H-function  under each choice of the contour $\mathcal{L}$ have been thoroughly considered in the book \cite{AA}. Let $z>0$ and under  conditions:
$$\sum_{i=1}^p A_j=\sum_{j=1}^q B_j,\;\;\rho\leq 1,$$
we get that the function $H_{q,p}^{p,0}(z)$ exists by means of \cite[Theorem 1.1]{AA}, if we choose $\mathcal{L}=\mathcal{L}_+$ or $\mathcal{L}=\mathcal{L}_{ic}$ under the additional restriction $\mu> 1.$ Only the second choice of the contour ensures the existence of the Mellin transform of $H_{q,p}^{p,0}(z)$, see \cite[Theorem 2.2]{AA}. In \cite[Theorem 6]{Karp}, the authors extend the condition $\mu>1$ to $\mu>0$ and proved that the function $H_{q,p}^{p,0}(z)$ is of compact support.\\

\begin{theorem}\label{TH1} Let $\alpha,\beta>-1.$ Assume that $\mu>0,\;\;\textrm{and}\;\;\sum_{j=1}^pA_j=\sum_{k=1}^qB_k.$
Then the following integral representation for the generalized Volterra function \eqref{DefGenVolt}
\begin{equation}\label{Int1}
V_{p,q}^{\alpha,\beta}\Big[_{(\beta_q,B_q)}^{(a_p,A_p)}\Big|x \Big]=\int_0^\rho H_{q,p}^{p,0}\left(t\Big|^{(B_q,\beta_q)}_{(A_p,\alpha_p)}\right)\frac{x^\alpha dt}{t\log^{\beta+1}(1/(tx))},
\end{equation}
holds true for all $x\in(0,1).$
\end{theorem}

\begin{proof}By using the Mellin transform  for the Fox's H--function $H_{q,p}^{p,0}(z)$ \cite[Theorem 6]{Karp}:
\begin{equation}\label{eq1}
\frac{\prod_{i=1}^p\Gamma(A_i t+\alpha_i)}{\prod_{k=1}^q \Gamma(B_k t+ \beta_k)}=\int_0^\rho H_{q,p}^{p,0}\left(z\Big|^{(B_q,\beta_q)}_{(A_p,\alpha_p)}\right)z^{t-1}dt,\;\Re(t)>\gamma,
\end{equation}
we obtain
\begin{equation}
\begin{split}
V_{p,q}^{\alpha,\beta}\Big[_{(\beta_q,B_q)}^{(a_p,A_p)}\Big|x \Big]&=\int_0^\infty \frac{\prod_{i=1}^p\Gamma(A_it+a_i)}{\prod_{j=1}^q\Gamma(B_jt+b_j)}\frac{x^{t+\alpha}t^\beta}{\Gamma(\beta+1)} dt\\
&=\int_0^\infty\int_0^\rho H_{q,p}^{p,0}\left(z\Big|^{(B_q,\beta_q)}_{(A_p,\alpha_p)}\right)\frac{ x^{t+\alpha}z^{t-1} t^\beta}{\Gamma(\beta+1)}dtdz\\
&=\int_0^\rho H_{q,p}^{p,0}\left(z\Big|^{(B_q,\beta_q)}_{(A_p,\alpha_p)}\right)\left(\int_0^\infty x^{t}z^{t} t^\beta dt\right)\frac{ x^{\alpha}z^{-1}}{\Gamma(\beta+1)}dz\\
&=\int_0^\rho H_{q,p}^{p,0}\left(z\Big|^{(B_q,\beta_q)}_{(A_p,\alpha_p)}\right)\left(\int_0^\infty  t^\beta e^{-t} dt\right)\frac{ x^{\alpha}z^{-1}}{\Gamma(\beta+1)\log^{\beta+1}\left(1/(xz)\right)}dz\\
&=\int_0^\rho H_{q,p}^{p,0}\left(z\Big|^{(B_q,\beta_q)}_{(A_p,\alpha_p)}\right)\frac{ x^{\alpha}z^{-1}}{\log^{\beta+1}\left(1/(xz)\right)}dz.
\end{split}
\end{equation}
This completes the proof of Theorem \ref{TH1}.
\end{proof}

\begin{remark} The special case for which the  H--function reduces to the Meijer G--function is when $A_1=...=A_p=B_1=...=B_q=A,\;A>0.$ In this case,
\begin{equation}\label{,,,}
H_{q,p}^{m,n}\left(z\Big|^{(A,b_q)}_{(A,a_p)}\right)=\frac{1}{A}G_{p,q}^{m,n}\left(z^{1/A}\Big|^{{\bf b}_\textbf{q}}_{\textbf{a}_\textbf{p}}\right),
\end{equation}
where $\textbf{a}_\textbf{p}=(a_1,...,a_p)$ and $\textbf{b}_\textbf{q}=(b_1,...,b_q).$ So we get
\begin{equation}\label{MPM}
V_{p,p}^{\alpha,\beta}\Big[_{(\beta_p,A)}^{(a_p,A)}\Big|x \Big]=\int_0^1G_{p,p}^{p,0}\left(t^{1/A}\Big|^{{\bf b}_\textbf{p}}_{\textbf{a}_\textbf{p}}\right) \frac{x^\alpha dt}{A t\log^{\beta+1}(1/(tx))},
\end{equation}
for all $x\in(0,1)$ and $\alpha,\beta>-1.$
\end{remark}

Let us  note that the special case of the Meijer G--function $G_{p,p}^{p,0}(\cdot)$ from \eqref{MPM} is very important for applications. Due to it in \cite{Karp2} it was proposed to name the function $G_{p,p}^{p,0}(\cdot)$ as Meijer--N{\o}rlund one, due to important results of N.E. N{\o}rlund for this function.

We denote the ratio of gamma--functions by
$$\psi_{n,m}=\frac{\prod_{i=1}^p\Gamma(\alpha_i+(n+m)A_i)}{\prod_{j=1}^q\Gamma(\beta_j+(n+m)B_j)},\;n,m\in\mathbb{N}_0.$$
In \cite[Corollary 1]{Khaled} it was proved that the function $H_{q,p}^{p,0}(z)$ is non--negative on $(0,\rho)$ if $$(H_1^n): \psi_{n,2}<\psi_{n,1}\;\textrm{and}\;\psi_{n,1}^2<\psi_{n,0}\psi_{n,2},\;\textrm{for\;all}\;n\in\mathbb{N}_0.$$
In addition, it was proved that  the H--function $H_{p,p}^{p,0}\big[t|^{(A,\beta_p)}_{(A,\alpha_p)}\big]$ is non--negative, if
$$(H_2):\;0<\alpha_1\leq...\leq\alpha_p,\;0<\beta_1\leq...\leq\beta_p,\;\;\sum_{j=1}^k \beta_j-\sum_{j=1}^k \alpha_j\geq0,\;\textrm{for}\;k=1,...,p.$$

\begin{corollary}Suppose that   conditions  $(H_2)$ are satisfied. Then the function
$$A\mapsto V_{p,p}^{\alpha,\beta}\Big[_{(\beta_p,A)}^{(a_p,A)}\Big|x \Big],$$
is log--convex on $(0,\infty).$ Furthermore, the following Tur\'an type inequality
\begin{equation}\label{turan}
V_{p,p}^{\alpha,\beta}\Big[_{(\beta_p,A)}^{(a_p,A)}\Big|x \Big]V_{p,p}^{\alpha,\beta}\Big[_{(\beta_p,A+2)}^{(a_p,A+2)}\Big|x \Big]-\left(V_{p,p}^{\alpha,\beta}\Big[_{(\beta_p,A+1)}^{(a_p,A+1)}\Big|x \Big]\right)^2\geq0,
\end{equation}
holds true.
\end{corollary}

\begin{proof} Rewriting the integral representation (\ref{MPM}) in the following form
\begin{equation}\label{kol}
V_{p,p}^{\alpha,\beta}\Big[_{(\beta_p,A)}^{(a_p,A)}\Big|x \Big]=\int_0^1G_{p,p}^{p,0}\left(t\Big|^{{\bf b}_\textbf{p}}_{\textbf{a}_\textbf{p}}\right) \frac{x^\alpha dt}{t\left(\log(1/x)+A\log(1/t)\right)^{\beta+1}},
\end{equation}
let us recall the Rogers--H\"older--Riesz inequality \cite[p. 54]{DM}, that is
\begin{equation}
\int_a^b|f(t)g(t)|dt\leq\left[\int_a^b|f(t)|^pdt\right]^{1/p}\left[\int_a^b|g(t)|^pdt\right]^{1/q},
\end{equation}
where $p\geq1,\;\frac{1}{p}+\frac{1}{q}=1,\;f$ and $g$ are real functions defined on $(a,b)$ and $|f|^p,\;|g|^q$ are integrable functions on $(a,b)$. From the Rogers--H\"older--Riesz inequality again and integral representation (\ref{kol})  using the fact that the function $A\mapsto \frac{1}{(a+bA)^{\beta+1}},\;a\geq0,b>0$ is log--convex on $(0,\infty)$ we derive that for $A_1,A_2>0$ and $\lambda\in[0,1],$
\newpage
$$V_{p,p}^{\alpha,\beta}\Big[_{(\beta_p,\lambda A_1+(1-\lambda)A_2)}^{(a_p,\lambda A_1+(1-\lambda)A_2)}\Big|x \Big]=$$
\begin{equation*}
\begin{split}
\;\;\;\;\;\;\;\;\;\;\;\;\;\;\;\;\;\;\;\;\;\;\;\;\;\;\;\;\;\;&=\int_0^1G_{p,p}^{p,0}\left(t\Big|^{{\bf b}_\textbf{p}}_{\textbf{a}_\textbf{p}}\right) \frac{x^\alpha dt}{t\left(\log(1/x)+\log(1/t)(A_1+(1-\lambda)A_2)\right)^{\beta+1}}\\
&\leq\int_0^1G_{p,p}^{p,0}\left(t\Big|^{{\bf b}_\textbf{p}}_{\textbf{a}_\textbf{p}}\right) \frac{x^\alpha dt}{t\left(\log(1/x)+A_1\log(1/t)\right)^{\lambda(\beta+1)}
\left(\log(1/x)+A_2\log(1/t)\right)^{(1-\lambda)(\beta+1)}}\\
&\leq\int_0^1\left[G_{p,p}^{p,0}\left(t\Big|^{{\bf b}_\textbf{p}}_{\textbf{a}_\textbf{p}}\right) \frac{x^\alpha }
{t\left(\log(1/x)+A_1\log(1/t)\right)^{\beta+1}}\right]^\lambda\\
&\cdot\left[G_{p,p}^{p,0}\left(t\Big|^{{\bf b}_\textbf{p}}_{\textbf{a}_\textbf{p}}\right) \frac{x^\alpha }{t\left(\log(1/x)+A_1\log(1/t)\right)^{\beta+1}}\right]^{1-\lambda}dt \\
&\leq\left[\int_0^1G_{p,p}^{p,0}\left(t\Big|^{{\bf b}_\textbf{p}}_{\textbf{a}_\textbf{p}}\right) \frac{x^\alpha dt}{t\left(\log(1/x)+A_1\log(1/t)\right)^{\beta+1}}\right]^\lambda\\
&\cdot\left[\int_0^1G_{p,p}^{p,0}\left(t\Big|^{{\bf b}_\textbf{p}}_{\textbf{a}_\textbf{p}}\right) \frac{x^\alpha dt}{t\left(\log(1/x)+A_2\log(1/t)\right)^{\beta+1}}\right]^{1-\lambda}\\
&=\left(V_{p,p}^{\alpha,\beta}\Big[_{(\beta_p,A_1)}^{(a_p,A_1)}\Big|x \Big]\right)^\lambda\left(V_{p,p}^{\alpha,\beta}\Big[_{(\beta_p,A_2)}^{(a_p,A_2)}\Big|x \Big]\right)^{1-\lambda}.
\end{split}
\end{equation*}
This implies that the function
$$A\mapsto V_{p,p}^{\alpha,\beta}\Big[_{(\beta_p,A)}^{(a_p,A)}\Big|x \Big]$$
is log--convex on $(0,\infty).$ Now let's go to the Tur\'an type inequality (\ref{turan}). Choosing  $A_1=A, A_2=A+2$ and $\lambda=\frac{1}{2}$ in the above inequality we get the desired result.
\end{proof}

\begin{corollary}Suppose that assumptions stated in Theorem \ref{TH1} and conditions $(H_1^n)$ are satisfied. Then the following inequality for the generalized Volterra function
\begin{equation}\label{EEE}
V_{p,q}^{\alpha,\beta_1}\Big[_{(b_q,B_q)}^{(a_p,A_p)}\Big|x \Big]V_{p,q}^{\alpha,\beta_2}\Big[_{(b_q,B_q)}^{(a_p,A_p)}\Big|x \Big]\leq\frac{\prod_{i=1}^p\Gamma(a_i)}{\prod_{j=1}^q\Gamma(b_i)}V_{p,q}^{\alpha,\beta_1+\beta_2}\Big[_{(b_q,B_q)}^{(a_p,A_p)}\Big|x \Big]
\end{equation}
holds true for all $0<x<1,\alpha,\beta_1,\beta_2>-1.$
\end{corollary}

\begin{proof}Recall the Chebyshev integral inequality \cite[p. 40]{DM}: if $f,g:[a,b]\longrightarrow\mathbb{R}$ are synchronous (both increasing  or decreasing) integrable functions, and $p:[a,b]\longrightarrow\mathbb{R}$  is a positive integrable function, then
\begin{equation}\label{OO}
\int_a^b p(t)f(t)dt\int_a^b p(t)g(t)dt\leq \int_a^b p(t)dt\int_a^b p(t)f(t)g(t)dt.
\end{equation}
Note that if $f$ and $g$ are asynchronous (one is decreasing and the other is increasing),
then (\ref{OO}) is reversed. Let $\beta_1,\beta_2>-1$ and  consider  functions $p,f,g:[0,\rho]\longrightarrow\mathbb{R}$ defined by:
$$p(t)= t^{-1}H_{q,p}^{p,0}\left(t\Big|^{(B_q,\beta_q)}_{(A_p,\alpha_p)}\right),\;\;f(t)=\frac{1}{(\log(1/x)+\log(1/t))^{\beta_1}},\;\;g(t)=\frac{1}{(\log(1/x)+\log(1/t))^{\beta_2}}.$$
Since the function $p$ is non--negative on $(0,\rho)$ and  functions $f$ and $g$ are increasing on
$(0,\rho),$ we conclude that the inequality (\ref{EEE}) holds true by means of the Chebyshev integral inequality (\ref{OO}) applied to the Mellin transform of the Fox's H--function \ref{eq1}.
\end{proof}

\begin{remark}Under the conditions $(H_2),$ the inequality (\ref{EEE}) reduces to the following inequality
\begin{equation}
V_{p,p}^{\alpha,\beta_1}\Big[_{(b_q,A)}^{(a_p,A)}\Big|x \Big]V_{p,p}^{\alpha,\beta_2}\Big[_{(b_p,A)}^{(a_p,A)}\Big|x \Big]\leq\frac{\prod_{i=1}^p\Gamma(a_i)}{\prod_{j=1}^q\Gamma(b_i)}V_{p,q}^{\alpha,\beta_1+\beta_2}\Big[_{(b_p,A)}^{(a_p,A)}\Big|x \Big]
\end{equation}
holds true for all $0<x<1,\alpha,\beta_1,\beta_2>-1.$
\end{remark}
Now let us note that Tur\'an type inequalities and connected results on log--convexity/log--concavity for different classes of special functions are very important and have many  applications, cf. \cite{Bar1,Bar2,Karp,KS1,KS2,SiMe1, SiMe2,SiMe3,SiMe4,Me1}.

Here, and in what follows, we denote the Laplace transform pair for a suitable function
$f$ as follows:
$$F(t)=L f(t),\;\;\textrm{and}\;\;f(t)=L^{-1}F(t),$$
that is,
$$L f(t)=\int_0^\infty e^{-xt}f(x)dx,\;\;\textrm{and}\;\;L^{-1}F(t)=\frac{1}{2i\pi}\int_{\textrm{Br}} e^{st} F(s)ds,$$
where Br denotes the Bromwich path. Recall that a function $f:(0,\infty)\longrightarrow(0,\infty)$ is called completely monotonic, if $f$ is  continuous on $[0,\infty)$, infinitely differentiable on $(0,\infty)$ and satisfies the following
inequality:
$$(-1)^n f^{(n)}(x)\geq0,\;\left(x>0,\;n\in\mathbb{N}_0=\left\{0,1,2,...\right\}\right).$$
The celebrated Bernstein Characterization Theorem gives a sufficient condition for the
complete monotonicity of a function $f$ in terms of the existence of some non-negative
locally integrable function $K( x )\;( x > 0)$, referred to as the spectral function, for which
$$f(s)=L(K)(s)=\int_0^\infty e^{-st}K(t)dt.$$

\begin{corollary} Suppose that conditions  of the Theorem \ref{TH1} are satisfied. In addition, assume that conditions $(H_1^n)$ are also valid. Then the function $\mathcal{V}_{p,q}^{\alpha,\beta}[x]$ defined by
\begin{equation}
\mathcal{V}_{p,q}^{\alpha,\beta}\Big[_{(b_q,B_q)}^{(a_p,A_p)}\Big|x \Big]=:V_{p,q}^{\alpha,\beta}\Big[_{(\beta_q,B_q)}^{(a_p,A_p)}\Big|e^{-x} \Big],
\end{equation}
is completely monotonic on $(0,\infty)$ for all $\alpha\geq0$ and $\beta>-1.$ Moreover, the function $$\mathcal{V}_{p,q}^{\alpha,\beta}\Big[_{(b_q,B_q)}^{(a_p,A_p)}\Big|x \Big]$$
is also completely monotonic on $(0,\infty)$  for each $\alpha\geq0$ and $\beta>-1$ under the hypothesis $(H_2).$
\end{corollary}

\begin{proof} By using the integral representation (\ref{Int1}), we can write the function $\mathcal{V}_{p,q}^{\alpha,\beta}[x]$ in the following form:
\begin{equation}
\mathcal{V}_{p,q}^{\alpha,\beta}\Big[_{(b_p,B_q)}^{(a_p,A_p)}\Big|x \Big]=e^{-\alpha x}\int_0^\rho H_{q,p}^{p,0}\left(t\Big|^{(B_q,\beta_q)}_{(A_p,\alpha_p)}\right)\frac{ dt}{t(x+\log(1/t))^{\beta+1}}
\end{equation}
So, the above representation reveals that $\mathcal{V}_{p,q}^{\alpha,\beta}[x]$ can be written as a product of two completely monotonic functions. This implies that $\mathcal{V}_{p,q}^{\alpha,\beta}[x]$ is completely monotonic on $(0,\infty)$  for each $\alpha\geq0$ and $\beta>-1.$
\end{proof}

Moreover, in \cite[Theorem 2]{Karp3}, Karp and  Prilepkina found the Mellin transform of the delta
neutral H function when $\mu=-m,\;m\in\mathbb{N}_0,$ that is
\begin{equation}\label{eA}
\frac{\prod_{i=1}^p\Gamma(A_i t+\alpha_i)}{\prod_{k=1}^q \Gamma(B_k t+ \beta_k)}=\int_0^\rho H_{q,p}^{p,0}\left(z\Big|^{(B_q,\beta_q)}_{(A_p,\alpha_p)}\right)z^{t-1}dt-\nu\rho^t\sum_{k=0}^m l_{m-k} t^k,\;\Re(t)>\gamma,
\end{equation}
where the coefficient $\nu$ is defined by
\begin{equation}\label{nu}
\nu=(2\pi)^{\frac{p-q}{2}}\prod_{i=1}^p A_i^{a_i-\frac{1}{2}}\prod_{j=1}^q B_j^{\frac{1}{2}-b_j},
\end{equation}
and the coefﬁcients $l_r$ satisfy the recurrence relation:
\begin{equation}\label{lr}
l_r=\frac{1}{r}\sum_{m=1}^r q_m l_{r-m},\;\;\textrm{with}\;l_0=1,
\end{equation}
with
$$q_m=\frac{(-1)^{m+1}}{m+1}\left[\sum_{i=1}^p\frac{\mathcal{B}_{m+1}(a_i)}{A_i^m}-\sum_{j=1}^p\frac{\mathcal{B}_{m+1}(b_j)}{B_j^m}\right],$$
where $\mathcal{B}_{m}$ is the Bernoulli polynomial defined via generating function \cite[p. 588]{QQ}
$$\frac{te^{at}}{e^t-1}=\sum_{n=0}^\infty\mathcal{B}_{n}(a)\frac{t^n}{n!},\;\;|t|<2\pi.$$

Obviously, by repeating the same calculations in Theorem \ref{TH1} with (\ref{eA}) and use the following known formula
\begin{equation}\label{YYY}
\int_0^\infty t^\sigma e^{-\rho t}dt=\frac{\Gamma(\sigma+1)}{\rho^{\sigma+1}},
\end{equation}
we can deduce the following result:
\begin{theorem}Let $\alpha,\beta>-1, \mu=-m,\;m\in\mathbb{N}_0\;\;\textrm{and}\;\;\sum_{j=1}^pA_j=\sum_{k=1}^qB_k.$
Then the generalized Volterra function possesses the following integral representation:
\begin{equation}
V_{p,q}^{\alpha,\beta}\Big[_{(\beta_q,B_q)}^{(a_p,A_p)}\Big|x \Big]=\int_0^\rho H_{q,p}^{p,0}\left(t\Big|^{(B_q,\beta_q)}_{(A_p,\alpha_p)}\right)\frac{x^\alpha dt}{t\log^{\beta+1}(1/(tx))}-\frac{\nu x^\alpha}{\Gamma(\beta)}\sum_{k=0}^m \frac{l_{m-k}\Gamma(\beta+k+1)}{\left[\log(1/(x\rho))\right]^{\beta+k+1}},
\end{equation}
where $0<x<1,$ the coefﬁcients $l_r$ are computed by (\ref{lr}), $\nu$ and $\rho$ are defined in (\ref{nu}) and (\ref{rho}).
\end{theorem}

In the next Theorem, we derive a Laplace type integral expression for the function $t^{\lambda-1}V_{p,p}^{\alpha,\beta}[.].$

\begin{theorem}\label{TT}Let $\alpha,\beta>-1$ and $\lambda\geq0.$ Then, the function
$$V_{p+1,p}^{\alpha+\lambda,\beta}\Big[_{\;\;\;\;(\beta_q,B_q)}^{(a_p,A),(\alpha+\lambda,1)}\Big|\frac{1}{x}\Big]=x^{-\lambda}V_{p+1,p}^{\alpha,\beta}\Big[_{\;\;\;\;(\beta_q,B_q)}^{(a_p,A),(\alpha+\lambda,1)}\Big|\frac{1}{x}\Big],$$
possesses the following integral representation
\begin{equation}\label{RRR}
V_{p+1,p}^{\alpha+\lambda,\beta}\Big[_{\;\;\;\;(\beta_q,B_q)}^{(a_p,A_p),(\alpha+\lambda,1)}\Big|\frac{1}{x}\Big]=\int_0^\infty e^{-xt} t^{\lambda-1} V_{p,q}^{\alpha,\beta}\Big[_{(\beta_q,B_q)}^{(a_p,A_p)}\Big|t\Big]dt,\;x>0.
\end{equation}
Moreover the function $V_{p+1,q}^{\alpha+\lambda,\beta}\Big[_{\;\;\;\;(b_q,B_q)}^{(a_p,A_p),(\alpha+\lambda,1)}\Big|\frac{1}{x}\Big]$ is completely monotonic on $(0,\infty).$
\end{theorem}
\begin{proof}Make use the formula (\ref{YYY})
and straightforward calculation  yields
\begin{equation*}
\begin{split}
\int_0^\infty e^{-xt} t^{\lambda-1} V_{p,q}^{\alpha,\beta}\Big[_{(\beta_q,B_q)}^{(a_p,A_p)}\Big|t\Big]dt&=\int_0^\infty e^{-xt} t^{\lambda-1}\left[\int_0^\infty \frac{\prod_{i=1}^p\Gamma(A_i s+a_i)}{\prod_{j=1}^q\Gamma(B_js+b_j)}\frac{t^{s+\alpha}s^\beta}{\Gamma(\beta+1)}ds\right]dt\\
&=\int_0^\infty \frac{\prod_{i=1}^p\Gamma(A_i s+a_i)}{\prod_{j=1}^q\Gamma(B_js+b_j)}\frac{s^\beta}{\Gamma(\beta+1)}\left[\int_0^\infty e^{-xt} t^{s+\alpha+\lambda-1}dt\right]ds\\
&=\int_0^\infty \frac{\Gamma(s+\alpha+\lambda)\prod_{i=1}^p\Gamma(A_i s+a_i)}{\prod_{j=1}^q\Gamma(B_js+b_j)}\frac{s^\beta x^{-(s+\alpha+\lambda)}}{\Gamma(\beta+1)}ds\\
&=V_{p+1,p}^{\alpha+\lambda,\beta}\Big[_{\;\;\;\;(\beta_q,B_q)}^{(a_p,A_p),(\alpha+\lambda,1)}\Big|\frac{1}{x}\Big].
\end{split}
\end{equation*}
Since the spectral function $t^{\lambda-1}V_{p,q}^{\alpha,\beta}[t]$  being positive, all prerequisites of the Bernstein Characterization Theorem for the completely monotone functions are fulfilled, that is the function $V_{p+1,q}^{\alpha+\lambda,\beta}\Big[_{\;\;\;\;(\beta_q,B_q)}^{(a_p,A_p),(\alpha+\lambda,1)}\Big|\frac{1}{x}\Big]$ is completely monotone in the above--mentioned range of the parameters involved.
\end{proof}

For some difficulties and wide--spread errors concerning generalization of the Bernstein Characterization Theorem to absolutely monotonic functions cf. \cite{Sit1}.

\begin{theorem}\label{T3}Let
$\alpha,\beta>-1,\eta>0.$
Then, we get
\begin{equation}\label{RTT}
\begin{split}
L^{-1}\left\{\frac{\log(\xi)}{\xi}V_{p,q-1}^{\alpha,\beta}\Big[_{(b_{q-1},B_{q-1})}^{\;\;(a_p,A_p)}\Big|\frac{1}{\xi} \Big]\right\}(x)&=\int_0^\infty\frac{\prod_{i=1}^p\Gamma(a_i+tA_i)t^\beta x^{t+\alpha}\psi(t+\alpha+1)}{\Gamma(\beta+1)\Gamma(\alpha+1+t)\prod_{j=1}^{q-1}\Gamma(b_j+tB_j)}dt\\
&-\log(\xi)V_{p,q}^{\alpha,\beta}\Big[_{(b_{q-1},B_{q-1}),(\alpha+1,1)}^{\;\;\;\;\;(a_p,A_p)}\Big|x \Big],
\end{split}
\end{equation}
where $\psi(z)=\frac{\Gamma^\prime(z)}{\Gamma(z)}$ is the digamma function.
\end{theorem}
\begin{proof} We set $(b_q,B_q)=(\alpha+\eta+1,1)$ and  define the function $F(x)$ by
$$F(x)=\int_0^\infty V_{p,q}^{\alpha+\eta,\beta}\Big[_{(b_{p},B_{q})}^{(a_p,A_p)}\Big|x \Big]f(\eta)d\eta.$$
Keeping (\ref{YYY}) in mind, we get
\begin{equation}
\begin{split}
L(F)(\xi)&=\int_0^\infty\int_0^\infty e^{-\xi x}V_{p,q}^{\alpha+\eta,\beta}\Big[_{(b_p,B_q)}^{(a_p,A_p)}\Big|x \Big]f(\eta)d\eta dx\\
&=\int_0^\infty\int_0^\infty\int_0^\infty e^{-\xi x}\frac{\prod_{i=1}^p\Gamma(A_it+a_i)x^{t+\alpha+\eta}t^\beta}{\Gamma(\beta+1)\prod_{j=1}^q\Gamma(B_jt+b_j)}f(\eta)d\eta dx dt\\
&=\int_0^\infty\int_0^\infty\frac{\prod_{i=1}^p\Gamma(A_it+a_i)t^\beta}{\Gamma(\beta+1)\prod_{j=1}^{q-1}\Gamma(B_jt+b_j)}\left(\frac{1}{\xi}\right)^{t+\alpha+\eta+1}f(\eta)d\eta dt\\
&=\xi^{-1}V_{p,q-1}^{\alpha,\beta}\Big[_{(b_{q-1},B_{q-1})}^{(a_p,A_p)}\Big|1/\xi\Big]\int_0^\infty\left(\frac{1}{\xi}\right)^{\eta}f(\eta)d\eta\\
&=V_{p,q-1}^{\alpha,\beta}\Big[_{(b_{q-1},B_{q-1})}^{(a_p,A_p)}\Big|\frac{1}{\xi}\Big]\frac{Lf(\log(\xi))}{\xi}.
\end{split}
\end{equation}
This implies that
\begin{equation}\label{MMM}
F(x)=L^{-1}\left\{V_{p,q-1}^{\alpha,\beta}\Big[_{(b_{q-1},B_{q-1})}^{\;\;\;(a_p,A_p)}\Big|\frac{1}{\xi}\Big]\frac{Lf(\log(\xi))}{\xi}\right\}(x).
\end{equation}
Now, suppose that $f(\eta)=\delta^\prime(\eta),$ where $\delta$ is the Dirac delta function. Since $L(\delta^\prime(\xi))=\xi,$ we deduce, by the above formula, that
\begin{equation}\label{RT}
L^{-1}\left\{V_{p,q-1}^{\alpha,\beta}\Big[_{(b_{q-1},B_{q-1})}^{\;\;\;(a_p,A_p)}\Big|\frac{1}{x} \Big]\frac{Lf(\log(\xi))}{\xi}\right\}(\xi)=\int_0^\infty V_{p,q}^{\alpha+\eta,\beta}\Big[_{(b_p,B_q)}^{(a_p,A_p)}\Big|x \Big]\delta^\prime(\eta)d\eta.
\end{equation}
Combining (\ref{RT})with the following formula
$$\int_0^\infty f(t)\delta^{(n)}(t)dt=(-1)^nf^{(n)}(0),$$
we obtain
\begin{equation}\label{RTT4}
L^{-1}\left\{V_{p,q-1}^{\alpha,\beta}\Big[_{(b_{q-1},B_{q-1})}^{\;\;(a_p,A_p)}\Big|\frac{1}{x} \Big]\frac{Lf(\log(x))}{x}\right\}(\xi)=-\lim_{\eta\rightarrow0}\frac{\partial}{\partial \eta}V_{p,q}^{\alpha+\eta,\beta}\Big[_{(b_p,B_q)}^{(a_p,A_p)}\Big|\xi \Big].
\end{equation}
Moreover, we have
\begin{equation}\label{ETY}
\begin{split}
\lim_{\eta\rightarrow0}\frac{\partial}{\partial \eta}V_{p,q}^{\alpha+\eta,\beta}\Big[_{(b_p,B_q)}^{(a_p,A_p)}\Big|\xi \Big]&=\log(\xi)V_{p,q}^{\alpha,\beta}\Big[_{(b_{q-1},B_{q-1}),(\alpha+1,1)}^{\;\;\;\;\;(a_p,A_p)}\Big|\xi \Big]\\&-\int_0^\infty\frac{\prod_{i=1}^p\Gamma(a_i+tA_i)t^\beta \xi^{t+\alpha}\psi(t+\alpha+1)}{\Gamma(\beta+1)\Gamma(\alpha+1+t)\prod_{j=1}^{q-1}\Gamma(b_j+tB_j)}dt
\end{split}
\end{equation}
In view of (\ref{RTT4}) and (\ref{ETY}) we obtain the desired result. The proof of Theorem \ref{T3} is complete.
\end{proof}
\begin{remark} let $q=p+1, a_i=b_i$ and  $A_i=B_i$ for $i=1,...,p$ in (\ref{RTT}), we obtain
\begin{equation}
L^{-1}\left\{\frac{1}{\xi^{\alpha+1}\log^{\beta}(\xi)}\right\}(x)=\int_0^\infty\frac{t^\beta x^{t+\alpha}\psi(t+\alpha+1)}{\Gamma(\beta+1)\Gamma(\alpha+1+t)}dt-\log(x)\mu(x,\beta,\alpha).
\end{equation}
In particular, for $\beta=0,$ we find \cite[Eq. (20)]{App3}
\begin{equation}
\frac{x^\alpha}{\Gamma(\alpha+1)}=\int_0^\infty \frac{x^{t+\alpha}\psi(t+\alpha+1)}{\Gamma(t+\alpha+1)}dt-\log(x) \nu(x,\alpha),\;\alpha>-1.
\end{equation}
\end{remark}
\begin{corollary} Suppose that
$$\alpha\geq0,\beta>-1,\;(a_p,A_p)=(\alpha,1)\;\textrm{and}\;\;(b_q,B_q)=(\alpha+\eta+1,1).$$
Then the next convolution representation holds
\begin{equation}\label{DDD}
\int_0^\infty V_{p,q}^{\alpha+\eta,\beta}\Big[_{(b_{p},B_{q})}^{(a_p, A_p)}\Big|x \Big]f(\eta)d\eta=f(\log(x))* \left(\frac{1}{x}V_{p-1,q-1}^{\alpha,\beta}\Big[_{(b_{p-1},B_{q-1})}^{(a_{p-1},A_{p-1})}\Big|x \Big]\right).
\end{equation}
\end{corollary}
\begin{proof}By means of Theorem \ref{TT}, we have
\begin{equation}
L\left\{x^{-1}V_{p,q}^{\alpha,\beta}\Big[_{(b_{q},B_{q})}^{(a_{p},A_{p})}\Big|\frac{1}{x} \Big]\right\}(s)=V_{p+1,q}^{\alpha,\beta}\Big[_{(b_{q},B_{q})}^{(a_{p},A_{p}),(\alpha,1)}\Big|\frac{1}{s}\Big].
\end{equation}
So, from the above formula and (\ref{MMM}) we conclude that (\ref{DDD}) holds.
\end{proof}

\section{Mathieu--type series associated with the generalized Volterra function}

Our aim in this section is to derive some integral representations  for a family of convergent Mathieu--
type series and its alternating variants with terms containing generalized Volterra functions.

Throughout this section, we adopt the following notation for the real sequence {\bf c}:
\begin{equation}\label{cccc}
{\bf c}: 0<c_1<...<c_n \uparrow \infty.
\end{equation}
It is useful here to consider the function $c:\mathbb{R}_+\rightarrow\mathbb{R}_+$ such that
$$c(x)\Big|_{x\in\mathbb{N}}={\bf c}.$$

In this section, we investigate the Mathieu--type series $\mathcal{K}$ and its alternating variant $\tilde{\mathcal{K}}$, which are defined by
\begin{equation}
\mathcal{K}(V_{p+1,q}^{\alpha,\beta};\textbf{c};r)=\sum_{j=1}^\infty\frac{c_j^{-\lambda}}{(c_j+r)^\mu}.V_{p+1,q}^{\alpha,\beta}\Big[_{\;\;\;\;(b_p,B_q)}^{(a_{p},A_{p}),(\alpha+\lambda,1)}\Big|\frac{r}{c_j} \Big],
\end{equation}
and
\begin{equation}
\tilde{\mathcal{K}}(V_{p+1,q}^{\alpha,\beta};\textbf{c};r)=\sum_{j=1}^\infty\frac{(-1)^{j-1}c_j^{-\lambda}}{(c_j+r)^\mu}.V_{p+1,q}^{\alpha,\beta}\Big[_{\;\;\;(b_p,B_q)}^{(a_{p},A_{p}),(\alpha+\lambda,1)}\Big|\frac{r}{c_j} \Big].
\end{equation}

The Laplace integral form of a Dirichlet series was one of most powerful tools in getting closed integral form expressions for  the Mathieu--type series $\mathcal{K}$ and its alternating variant $\tilde{\mathcal{K}}.$ For $\textbf{c}$ satisfying (\ref{cccc}), we have
\begin{equation}
\mathcal{D}_{\textbf{c}}(x)=\sum_{n=1}^\infty e^{-c_n^\eta x}=x\int_0^\infty e^{-xt}A_\eta(t)dt,
\end{equation}
where the so--called counting function $A_\eta( t )$ has been found easily in the following manner
$$A_\eta(t)=\sum_{n:a_n^\eta\geq t}1=[c^{-1}(t^{\frac{1}{\eta}})],$$
where $c^{-1}(t)$ is the inverse function of $c(x),$ and $[a]$ is the integer part of a real number $a.$ From this, and using the fact that
$$[c^{-1}(t^{\frac{1}{\eta}})]\equiv 0,\;t\in[0,c_1^\eta),$$
we find that
\begin{equation}\label{iii}
\mathcal{D}_{\textbf{c}}(x)=x\int_{c_1^\eta}^\infty e^{-xt}[c^{-1}(t^{\frac{1}{\eta}})]dt.
\end{equation}
 A comprehensive overview of this technique is in \cite{TI1, TI3}. Note important results on inequalities for Mathieu series proved by V.P.Zastavnyi in \cite{Zas1,Zas2} and Z. Tomovski et al. in \cite{Z1,Z2}.

\begin{theorem}\label{T9}Let $\alpha,\beta>-1$ and $\lambda,\mu,r>0.$ Then for the Mathieu--type power series $\mathcal{K}(V_{p+1,q}^{\alpha,\beta};\textbf{c};r)$  the next integral representation is valid:
\begin{equation}
\mathcal{K}(V_{p+1,q}^{\alpha,\beta};\textbf{c};r)=K^V_{\bf c}(r,\lambda,\mu+1)+\mu K^V_{\bf c}(r,\lambda+1,\mu),
\end{equation}
where
\begin{equation}
K^V_{\bf c}(r,\lambda,\mu)=\int_{c_1}^\infty\frac{[c^{-1}(x)]}{x^\lambda(x+r)^\mu}V_{p+1,q}^{\alpha,\beta}\Big[_{\;\;\;\;(b_p,B_q)}^{(a_{p},A_{p}),(\alpha+\lambda,1)}\Big|\frac{r}{x} \Big]dx.
\end{equation}
\end{theorem}
\begin{proof}Consider the Laplace transform formula (\ref{RRR}) of the function $x^{\lambda-1}V_{p,q}^{\alpha,\beta}[.]$  and in view of  (\ref{YYY}) and (\ref{iii}), we get
\begin{equation}
\begin{split}
\mathcal{K}(V_{p+1,q}^{\alpha,\beta};\textbf{c};r)&=\sum_{j=1}^\infty\frac{1}{(c_j+r)^\mu}\int_0^\infty e^{-c_js}s^{\lambda-1}V_{p,q}^{\alpha,\beta}\Big[_{(b_p,B_q)}^{(a_{p},A_{p})}\Big|rs \Big]dx\\
&=\frac{1}{\Gamma(\mu)}\sum_{j=1}^\infty\int_0^\infty\int_0^\infty e^{-c_js-(c_j+r)t}t^{\mu-1}s^{\lambda-1}V_{p,q}^{\alpha,\beta}\Big[_{(b_p,B_q)}^{(a_{p},A_{p})}\Big|rs \Big]dxdt\\
&=\frac{1}{\Gamma(\mu)}\int_0^\infty\int_0^\infty \left(\sum_{j=1}^\infty e^{-c_j(s+t)}\right)e^{-rt}t^{\mu-1}s^{\lambda-1}V_{p,q}^{\alpha,\beta}\Big[_{(b_p,B_q)}^{(a_{p},A_{p})}\Big|rs \Big]dxdt\\
&=\int_{c_1}^\infty \frac{[c^{-1}(x)]}{\Gamma(\mu)}\left(\int_0^\infty e^{-(r+x)t}t^{\mu-1}dt\right)\left(\int_0^\infty e^{-xs}s^\lambda V_{p,q}^{\alpha,\beta}\Big[_{(b_p,B_q)}^{(a_{p},A_{p})}\Big|rs \Big]ds\right)dx\\
&+\int_{c_1}^\infty \frac{[c^{-1}(x)]}{\Gamma(\mu)}\left(\int_0^\infty e^{-(r+x)t}t^{\mu}dt\right)\left(\int_0^\infty e^{-xs}s^{\lambda-1} V_{p,q}^{\alpha,\beta}\Big[_{(b_p,B_q)}^{(a_{p},A_{p})}\Big|rs \Big]ds\right)dx\\
&=\int_{c_1}^\infty \frac{[c^{-1}(x)]}{x^{\lambda+1}(r+x)^\mu}V_{p+1,q}^{\alpha,\beta}
\Big[_{(b_q,B_q)}^{(a_{p},A_{p}),(\alpha+\lambda+1,1)}
\Big|\frac{r}{x}\Big]dsdx
\\
&+\mu\int_{c_1}^\infty \frac{[c^{-1}(x)]}{x^{\lambda}(r+x)^{\mu+1}}V_{p+1,q}^{\alpha,\beta}
\Big[_{(b_q,B_q)}^{(a_{p},A_{p}),(\alpha+\lambda,1)}
\Big|\frac{r}{x}\Big]ds
dx,
\end{split}
\end{equation}
which proves the Theorem \ref{T9}.
\end{proof}

Obviously, by repeating the same calculations as above and using the formula (see \cite[Eq. (13)]{TI3})
\begin{equation}
\tilde{D}_{\textbf{c}}(y)=\sum_{j=1}^\infty(-1)^{j-1}e^{-c_j y}=y\int_{c_1}^\infty  e^{-yx}\sin^2\left(\frac{\pi}{2}\left[c^{-1}(x)\right]\right)dx,
\end{equation}
we can deduce the following result for the alternating Mathieu--type power series $\tilde{\mathcal{K}}.$

\begin{theorem}Let $\alpha,\beta>-1$ and $\lambda,\mu,r>0.$ Then for the alternating Mathieu--type power series $\tilde{\mathcal{K}}(V_{p+1,q}^{\alpha,\beta};\textbf{c};r)$  the next integral representation is valid:
\begin{equation}
\tilde{\mathcal{K}}(V_{p+1,q}^{\alpha,\beta};\textbf{c};r)=\tilde{K}^V_{\bf c}(r,\lambda,\mu+1)+\mu \tilde{K}^V_{\bf c}(r,\lambda+1,\mu),
\end{equation}
where
\begin{equation}
\tilde{K}^V_{\bf c}(r,\lambda,\mu)=\int_{c_1}^\infty\frac{\sin^2\left(\frac{\pi}{2}[c^{-1}(x)]\right)}{x^\lambda(x+r)^\mu}V_{p+1,q}^{\alpha,\beta}\Big[_{\;\;\;\;(b_p,B_q)}^{(a_{p},A_{p}),(\alpha+\lambda,1)}\Big|\frac{r}{x} \Big]dx.
\end{equation}
\end{theorem}
\vspace{0.5cm}

\noindent\textbf{Acknowledgment}

The authors like to thank Prof. Alexander Apelblat for providing us with the copy of his book on
Volterra functions, it was very useful for the preparation of this paper.

\end{document}